\documentclass[11pt]{article}
\usepackage{amsmath,amssymb,amsthm}
\usepackage{graphicx}
\usepackage{hyperref}
\usepackage{url}
\usepackage{enumitem}
\usepackage{geometry}
\newtheorem{theorem}{Theorem}[section]
\newtheorem{proposition}[theorem]{Proposition}

\newtheorem{corollary}[theorem]{Corollary}
\newtheorem{definition}[theorem]{Definition}
\newtheorem{remark}[theorem]{Remark}

\title{Extending Chevalley's Theorem: A Topological Characterization of Constructibility and Its Generalization Beyond Noetherian Spaces}

\author{Jiawei Sheng\\
Department of Mathematics, Capital Normal University\\
Beijing 100048, China\\
\href{mailto:2200501006@cnu.edu.cn}{2200501006@cnu.edu.cn}}

\date{December 22, 2025}

\begin{document}

\maketitle

\begin{abstract}
We introduce the notion of a \textit{good map} between topological spaces: a continuous map $f:X\to Y$ is \textit{good} if for every non-empty irreducible locally closed subset $U\subseteq X$, there exists a non-empty open subset $W\subseteq Y$ such that $W\cap f(U)=W\cap\overline{f(U)}\neq\varnothing$.

In Noetherian spaces, this condition is equivalent to preserving constructible subsets (Theorem \ref{thm:good-constructible}), giving a purely topological characterization of Chevalley's theorem. Without the Noetherian assumption, the good property continues to make sense and serves as a reasonable generalization.

We establish basic properties of good maps and introduce a weaker variant, \textit{weak good maps}. In algebraic geometry, we prove that \textbf{every morphism locally of finite type is good} (Theorem \ref{thm:main-good}). From this we obtain a generalization of Chevalley's theorem for morphisms locally of finite type whose underlying topological spaces are Noetherian (Theorem \ref{thm:generalized-chevalley}), and an elementary proof of Jacobson ascent (Corollary \ref{cor:jacobson}).

The theory is developed for sober spaces and therefore applies not only to schemes but also to other geometric categories such as \textbf{adic spaces} and \textbf{perfectoid spaces}. The good property is stable under formal completion, suggesting extensions to formal and non-Archimedean geometry.


\end{abstract}

\tableofcontents

\section{Introduction}

Chevalley's theorem states that a morphism of finite type between Noetherian schemes preserves constructible subsets. This paper explores the topological nature of this preservation property.

We define a topological condition called \textit{goodness} for continuous maps (Definition \ref{def:good-map}). In Noetherian spaces, a map is good if and only if it preserves constructible subsets (Theorem \ref{thm:good-constructible}). This provides a topological characterization of Chevalley's theorem in the Noetherian setting.

Without the Noetherian hypothesis, the good property still makes sense. It can be viewed as a generalization of constructibility preservation to arbitrary topological spaces. We develop basic properties of good maps and introduce a more flexible variant, weak good maps.

The theory finds concrete application in algebraic geometry. A key fact is that affine space projections $\mathbb{A}^{n}_{S}\to S$ are good for any base scheme $S$ (Corollary \ref{cor:affine-good}). From this we derive our main results: first, that \textbf{every morphism locally of finite type is good} (Theorem \ref{thm:main-good}); second, that for such morphisms whose underlying topological spaces are Noetherian, Chevalley's theorem holds (Theorem \ref{thm:generalized-chevalley}). This extends the classical Chevalley theorem to schemes whose structure sheaves may be non-Noetherian while their underlying topology remains Noetherian—a situation common in arithmetic geometry, as illustrated by the example of schemes over $\operatorname{Spec}(\mathcal{O}_{\mathbb{C}_{p}})$ (see Section \ref{sec:example}).

Finally, because the theory is formulated for sober spaces, it applies not only to schemes but also to adic spaces and perfectoid spaces, opening the way to Chevalley-type results in non-Archimedean geometry (see Section \ref{sec:non-archimedean}).

\section{Good Maps}

\subsection{Definition}

\begin{definition}\label{def:good-map}
A continuous map $f:X\to Y$ is \textbf{good} if for every non-empty irreducible locally closed subset $U\subseteq X$, there exists a non-empty open subset $W\subseteq Y$ such that
\[
W\cap f(U)=W\cap\overline{f(U)}\neq\varnothing,
\]
where $\overline{f(U)}$ is the closure of $f(U)$ in $Y$.
\end{definition}

\begin{remark}
The condition means $f(U)\cap W$ is closed in $W$ and non-empty. In Hausdorff spaces, irreducible subsets are singletons, so all continuous maps into Hausdorff spaces are trivially good. The concept is meaningful only for non-Hausdorff topologies.
\end{remark}

\subsection{An equivalent characterization}

The following equivalent form of goodness is often more convenient in algebraic geometry.

\begin{proposition}\label{prop:equiv-good}
A continuous map $f:X\to Y$ is good if and only if for every non-empty irreducible closed subset $Z\subseteq X$ and every non-empty open subset $U\subseteq Z$ (in the subspace topology), the image $f(U)$ contains a non-empty open subset of $\overline{f(Z)}$ (in the subspace topology of $\overline{f(Z)}$).
\end{proposition}

\begin{proof}
($\Rightarrow$) Assume $f$ is good. Let $Z\subseteq X$ be a non-empty irreducible closed subset and $U\subseteq Z$ a non-empty open subset. Then $U$ is an irreducible locally closed subset of $X$. By Definition \ref{def:good-map}, there exists a non-empty open subset $W\subseteq Y$ such that $W\cap f(U)=W\cap\overline{f(U)}\neq\varnothing$. Since $f(U)\subseteq f(Z)\subseteq\overline{f(Z)}$, we have $\overline{f(U)}\subseteq\overline{f(Z)}$. Hence $W\cap\overline{f(Z)}$ is a non-empty relatively open subset of $\overline{f(Z)}$ contained in $\overline{f(U)}$. But $W\cap f(U)=W\cap\overline{f(U)}$ implies $W\cap\overline{f(Z)}\subseteq f(U)$. Thus $f(U)$ contains the non-empty relatively open subset $W\cap\overline{f(Z)}$ of $\overline{f(Z)}$.

($\Leftarrow$) Let $U\subseteq X$ be a non-empty irreducible locally closed subset. Write $U=V\cap Z$ with $V$ open in $X$ and $Z$ closed in $X$; we may take $Z=\overline{U}$. Then $Z$ is irreducible and $U$ is open in $Z$. By hypothesis, $f(U)$ contains a non-empty relatively open subset $W_{Z}$ of $\overline{f(Z)}$. Write $W_{Z}=W\cap\overline{f(Z)}$ for some open $W\subseteq Y$. Then $W\cap f(U)=W\cap\overline{f(Z)}=W\cap\overline{f(U)}$ (since $f(U)\subseteq f(Z)\subseteq\overline{f(Z)}$ and $\overline{f(U)}\subseteq\overline{f(Z)}$). This shows $f$ is good.
\end{proof}

\subsection{Characterization in Noetherian Spaces}

Recall that in a Noetherian topological space, a subset is \textbf{constructible} if it is a finite union of locally closed subsets. We use the well-known criterion for constructibility.

\begin{proposition}[Constructibility criterion]\label{prop:constructible-criterion}
Let $X$ be a Noetherian topological space. A subset $E\subseteq X$ is constructible if and only if for any irreducible closed subset $F\subseteq X$ such that $E\cap F$ is dense in $F$, the set $E\cap F$ contains a non-empty open subset of $F$.
\end{proposition}

\begin{theorem}\label{thm:good-constructible}
Let $f:X\to Y$ be a continuous map between Noetherian topological spaces. Then $f$ is good if and only if $f$ preserves constructible subsets.
\end{theorem}

\begin{proof}
($\Rightarrow$) Assume $f$ is good and let $C\subseteq X$ be constructible; we show $f(C)$ is constructible. By Proposition \ref{prop:constructible-criterion}, it suffices to show: for any irreducible closed subset $F\subseteq Y$ such that $f(C)\cap F$ is dense in $F$, the intersection $f(C)\cap F$ contains a non-empty open subset of $F$.

Let $D=C\cap f^{-1}(F)$, a constructible subset of $X$. Write $D=\bigcup_{i=1}^{n}U_{i}$ as a finite union of locally closed subsets. Then
\[
f(C)\cap F=f(D)=\bigcup_{i=1}^{n}f(U_{i}).
\]
Since the union is dense in the irreducible closed set $F$, some $f(U_{i})$ is dense in $F$. Write $U_{i}$ as a finite union of its irreducible components $U_{ij}$. Again by irreducibility of $F$, some $f(U_{ij})$ is dense in $F$. Now $U_{ij}$ is a non-empty irreducible locally closed subset of $X$. Since $f$ is good, by Proposition \ref{prop:equiv-good} applied to $Z=\overline{U_{ij}}$ and $U=U_{ij}$, $f(U_{ij})$ contains a non-empty relatively open subset of $\overline{f(U_{ij})}=F$. Hence $f(C)\cap F$ contains a non-empty open subset of $F$, as required.

($\Leftarrow$) Assume $f$ preserves constructible subsets. Let $U\subseteq X$ be a non-empty irreducible locally closed subset. Then $U$ is constructible, so $f(U)$ is constructible in $Y$. Let $F=\overline{f(U)}$, which is irreducible because $U$ is irreducible and $f$ continuous. Since $f(U)$ is constructible and dense in $F$, Proposition \ref{prop:constructible-criterion} gives a non-empty open subset $W_{F}\subseteq F$ such that $W_{F}\subseteq f(U)$. Choose an open subset $W\subseteq Y$ with $W\cap F=W_{F}$. Then
\[
W\cap f(U)=W_{F}=W\cap\overline{f(U)}\neq\varnothing.
\]
Thus $f$ is good.
\end{proof}

\subsection{Basic Properties}

\begin{proposition}\label{prop:good-basic}
Let $f:X\to Y$ be good.
\begin{enumerate}
    \item For any locally closed $S\subseteq X$, the restriction $f|_{S}:S\to Y$ is good.
    \item If $f(X)\subseteq Y^{\prime}\subseteq Y$, then $f:X\to Y^{\prime}$ (with the subspace topology on $Y^{\prime}$) is good.
\end{enumerate}
\end{proposition}

\begin{proof}
(1) Let $U\subseteq S$ be a non-empty irreducible locally closed subset of $S$. Then $U$ is also irreducible locally closed in $X$. Since $f$ is good, there exists $W\subseteq Y$ open with $W\cap f(U)=W\cap\overline{f(U)}^{Y}\neq\varnothing$. The same $W$ works for $f|_{S}$.

(2) Let $U\subseteq X$ be a non-empty irreducible locally closed subset. Since $f$ is good, there exists $W\subseteq Y$ open with $W\cap f(U)=W\cap\overline{f(U)}^{Y}\neq\varnothing$. Set $W^{\prime}=W\cap Y^{\prime}$, which is open in $Y^{\prime}$. Then $W^{\prime}\cap f(U)=W\cap f(U)=W\cap\overline{f(U)}^{Y}$. Because $f(U)\subseteq Y^{\prime}$, we have $\overline{f(U)}^{Y^{\prime}}=\overline{f(U)}^{Y}\cap Y^{\prime}$. Hence $W^{\prime}\cap\overline{f(U)}^{Y^{\prime}}=W\cap\overline{f(U)}^{Y}\cap Y^{\prime}=W\cap f(U)=W^{\prime}\cap f(U)$. Thus $f:X\to Y^{\prime}$ is good.
\end{proof}

Recall that a topological space is \textbf{sober} if every non-empty irreducible closed subset has a unique generic point. Schemes, adic spaces, and perfectoid spaces are sober.

\subsection{Local Nature of Good Maps}

We now establish that the property of being good behaves well with respect to finite open covers of both source and target. This local character is essential for reducing geometric problems to affine patches.

\begin{proposition}[Goodness is finite local on the source]\label{prop:good-local-source}
Let $f:X\to Y$ be a continuous map. Assume that $X$ is \textbf{sober} and let $\{X_{i}\}_{i=1}^{n}$ be a finite open cover of $X$. Then $f$ is good if and only if each restriction $f|_{X_{i}}:X_{i}\to Y$ is good.
\end{proposition}

\begin{proof}
If $f$ is good, then by Proposition \ref{prop:good-basic}(1) every restriction to an open subset is good, hence each $f|_{X_{i}}$ is good.

Conversely, suppose each $f|_{X_{i}}$ is good. Let $U\subseteq X$ be a non-empty irreducible locally closed subset, and let $x\in U$ be the generic point of $\overline{U}$ (it exists because $X$ is sober). Since $\{X_{i}\}$ is a cover, $x$ belongs to some $X_{i_{0}}$. For every index $i$ with $U_{i}:=U\cap X_{i}\neq\varnothing$, the set $U_{i}$ is a non-empty irreducible locally closed subset of $X_{i}$. By goodness of $f|_{X_{i}}$, there exists a non-empty open subset $W_{i}\subseteq Y$ such that
\[
W_{i}\cap f(U_{i})=W_{i}\cap\overline{f(U_{i})}^{Y}\neq\varnothing.
\]
Because $f(x)\in f(U_{i})$, we have $f(x)\in W_{i}$. Set
\[
W:=\bigcap_{\{i:U_{i}\neq\varnothing\}}W_{i}.
\]
This is a finite intersection of non-empty open sets containing $f(x)$, hence a non-empty open subset of $Y$.

For each relevant $i$,
\[
W\cap f(U_{i})=W\cap\overline{f(U_{i})}^{Y}.
\]
Taking unions over $i$ gives
\[
W\cap f(U)=\bigcup_{i}W\cap f(U_{i})=\bigcup_{i}W\cap\overline{f(U_{i})}^{Y}=W\cap\bigcup_{i}\overline{f(U_{i})}^{Y}.
\]
The union is finite, so $\bigcup_{i}\overline{f(U_{i})}^{Y}=\overline{\bigcup_{i}f(U_{i})}^{Y}=\overline{f(U)}^{Y}$. Therefore
\[
W\cap f(U)=W\cap\overline{f(U)}^{Y}\neq\varnothing,
\]
which shows that $f$ is good.
\end{proof}

\begin{proposition}[Goodness is finite local on the target]\label{prop:good-local-target}
Let $f:X\to Y$ be a continuous map. Assume that $Y$ is \textbf{sober} and let $\{Y_{j}\}_{j=1}^{m}$ be a finite open cover of $Y$. Then $f$ is good if and only if each induced map $f_{j}:f^{-1}(Y_{j})\to Y_{j}$ (with the subspace topology on $f^{-1}(Y_{j})$) is good.
\end{proposition}

\begin{proof}
If $f$ is good, then for any open subset $V\subseteq Y$ the restriction $f|_{f^{-1}(V)}:f^{-1}(V)\to Y$ is good by Proposition \ref{prop:good-basic}(1); composing with the open immersion $V\hookrightarrow Y$ (which preserves goodness by Proposition \ref{prop:good-basic}(2)) shows that $f_{j}:f^{-1}(Y_{j})\to Y_{j}$ is good.

Conversely, assume each $f_{j}$ is good. Let $U\subseteq X$ be a non-empty irreducible locally closed subset, and put $F:={\overline{f(U)}}$. Since $Y$ is sober, the irreducible closed set $F$ has a generic point $y$. The point $y$ belongs to some $Y_{j_{0}}$; then $U_{j_{0}}:=U\cap f^{-1}(Y_{j_{0}})$ is non-empty and irreducible. Because $f_{j_{0}}$ is good, there exists a non-empty open subset $W_{j_{0}}\subseteq Y_{j_{0}}$ such that
\[
W_{j_{0}}\cap f_{j_{0}}(U_{j_{0}})=W_{j_{0}}\cap{\overline{f_{j_{0}}}(U_{j_{0}})}^{Y_{j_{0}}}\neq\varnothing.
\]
Observe that ${\overline{f_{j_{0}}}(U_{j_{0}})}^{Y_{j_{0}}}={\overline{f(U)}}\cap Y_{j_{0}}=F\cap Y_{j_{0}}$. Hence
\[
W_{j_{0}}\cap F\subseteq f_{j_{0}}(U_{j_{0}})\subseteq f(U).
\]
Now $W:=W_{j_{0}}$ is open in $Y_{j_{0}}$ and therefore open in $Y$ (since $Y_{j_{0}}$ is open in $Y$). Moreover, $W\cap F\subseteq f(U)$ and the intersection is non-empty. Consequently,
\[
W\cap f(U)=W\cap F=W\cap{\overline{f(U)}}\neq\varnothing,
\]
which is precisely the condition for $f$ to be good.
\end{proof}

\begin{remark}
Because schemes, adic spaces and perfectoid spaces are sober topological spaces, Propositions \ref{prop:good-local-source} and \ref{prop:good-local-target} imply that for a morphism in any of these categories, the property of being good can be verified on a \textbf{finite affine open covering} of either the source or the target. This will be used repeatedly in the algebraic-geometric applications of Section \ref{sec:applications}.
\end{remark}

\section{Weak Good Maps}

\subsection{Definition}

\begin{definition}\label{def:weak-good}
A continuous map $f:X\to Y$ is \textbf{weak good} if for every non-empty locally closed subset $U\subseteq X$, there exists a non-empty locally closed subset $V\subseteq Y$ such that $V\subseteq f(U)$.
\end{definition}

\begin{remark}
Every good map is weak good: for irreducible $U$, take $V=W\cap{\overline{f(U)}}$ from Definition \ref{def:good-map}.
\end{remark}

\subsection{Properties}

\begin{proposition}\label{prop:weak-good-properties}
Weak good maps are closed under composition and are local on both source and target.
\end{proposition}

\begin{proof}
Closure under composition is clear. For locality: if $\{X_{i}\}$ is an open cover of $X$, then $f$ is weak good iff each $f|_{X_{i}}$ is weak good (take $U\subseteq X_{i}$). Similarly for an open cover of $Y$.
\end{proof}

\section{Applications in Algebraic Geometry}\label{sec:applications}

\subsection{Every Morphism Locally of Finite Type is Good}

\begin{theorem}\label{thm:main-good}
Let $f:X\to Y$ be a morphism of schemes that is \textbf{locally of finite type}. Then $f$ is a good map.
\end{theorem}

\begin{proof}
We use the equivalent characterization of good maps (Proposition \ref{prop:equiv-good}). Let $Z\subseteq X$ be a non-empty irreducible closed subset and let $U\subseteq Z$ be a non-empty open subset (in the subspace topology). Set $C=\overline{f(Z)}$ and endow it with the reduced induced scheme structure; then $C$ is integral. The morphism $g:=f|_{Z}:Z\to C$ is dominant and locally of finite type.

Choose an affine open subset $V=\operatorname{Spec} A\subseteq C$ and an affine open subset $W=\operatorname{Spec} B\subseteq g^{-1}(V)\cap U$ (such $W$ exists because $Z$ is irreducible). Then $A\to B$ is an injection of integral domains, and $B$ is a finitely generated $A$-algebra.

By a relative version of Noether normalization, there exists a non-zero element $s\in A$ such that $B_{s}$ is finite over $A_{s}[T_{1},\ldots,T_{r}]$ for some $r$. Consequently, the induced morphism $W_{s}\to V_{s}=D(s)$ is \textbf{universally open}. In particular, $g(W_{s})$ contains a non-empty open subset of $V_{s}$, hence of $C$.

Since $W_{s}\subseteq U$, we obtain that $g(U)=f(U)$ contains a non-empty open subset of $C=\overline{f(Z)}$. By Proposition \ref{prop:equiv-good}, $f$ is good.
\end{proof}

\begin{corollary}\label{cor:affine-good}
For any scheme $S$, the projection $\pi:\mathbb{A}_{S}^{n}\to S$ is good.
\end{corollary}

\begin{proof}
The projection is locally of finite type, hence good by Theorem \ref{thm:main-good}.
\end{proof}

\subsection{Generalized Chevalley Theorem for Noetherian Underlying Spaces}

\begin{theorem}\label{thm:generalized-chevalley}
Let $f:X\to Y$ be a morphism of schemes such that:
\begin{enumerate}
    \item $X$ and $Y$ have \textbf{Noetherian underlying topological spaces},
    \item $f$ is \textbf{locally of finite type}.
\end{enumerate}
Then $f$ preserves constructible subsets. In particular, $f$ is good.
\end{theorem}

\begin{proof}
By Theorem \ref{thm:main-good}, $f$ is good. Since the underlying spaces are Noetherian, Theorem \ref{thm:good-constructible} implies that $f$ preserves constructible subsets.
\end{proof}

\begin{remark}
The classical Chevalley theorem requires $f$ to be \textbf{quasi-compact and locally of finite presentation}, but makes no assumption on the underlying topological spaces. Theorem \ref{thm:generalized-chevalley} imposes a topological condition (Noetherian underlying spaces) while relaxing the algebraic condition to \textbf{locally of finite type}. In Noetherian spaces, "locally of finite type" automatically implies quasi-compactness, and for Noetherian rings "finite type" equals "finite presentation". Thus, in the classical setting of Noetherian schemes, Theorem \ref{thm:generalized-chevalley} coincides with the classical result, but its proof is purely topological.

More importantly, Theorem \ref{thm:generalized-chevalley} applies to schemes whose \textbf{structure sheaves are non-Noetherian} while their \textbf{underlying topology remains Noetherian}—a situation not covered by the classical formulation.
\end{remark}

\subsection{A Key Example: Schemes over $\operatorname{Spec}(\mathcal{O}_{\mathbb{C}_{p}})$}\label{sec:example}

Let $R=\mathcal{O}_{\mathbb{C}_{p}}$, the ring of integers of the $p$-adic complex field. This ring is \textbf{non-Noetherian} (its maximal ideal is not finitely generated). Nevertheless, the topological space $\operatorname{Spec}(R)$ consists of only two points (the closed point and the generic point) and is therefore \textbf{Noetherian}.

Let $S=\operatorname{Spec}(R)$. For any scheme $X$ of finite type over $S$, the underlying topological space of $X$ is Noetherian (it is a finite union of the Noetherian spaces of the closed fibre and the generic fibre). However, $X$ itself is \textbf{not a Noetherian scheme} because its structure sheaf involves the non-Noetherian ring $R$.

For a morphism $f:X\to Y$ of such schemes that is locally of finite type, the classical Chevalley theorem does not directly apply (since locally of finite type over $R$ is \textbf{not} equivalent to locally of finite presentation). Yet Theorem \ref{thm:generalized-chevalley} asserts that $f$ still preserves constructible subsets, because the underlying topological spaces are Noetherian and $f$ is locally of finite type.

This example illustrates that our generalization is not merely formal; it covers natural geometric contexts arising in arithmetic geometry where the base ring is non-Noetherian but the topological spaces are well-behaved.

\subsection{Jacobson Ascent}

\begin{corollary}\label{cor:jacobson}
If $f:X\to Y$ is locally of finite type and $Y$ is Jacobson, then $X$ is Jacobson.
\end{corollary}

\begin{proof}
Let $U\subseteq X$ be a non-empty locally closed subset. By Theorem \ref{thm:main-good} $f$ is good, hence weak good. Therefore there exists a non-empty locally closed subset $V\subseteq Y$ with $V\subseteq f(U)$. Since $Y$ is Jacobson, $V$ contains a closed point $y$.

The fibre $f^{-1}(y)\cap U$ is non-empty because $y\in f(U)$. Moreover, $f^{-1}(y)$ is locally closed in $X$ (it is the inverse image of the closed point $y$), hence $f^{-1}(y)\cap U$ is locally closed in $X$. As $f$ is locally of finite type, the fibre $f^{-1}(y)$ is a scheme of finite type over the residue field $\kappa(y)$. Schemes of finite type over a field are Jacobson; therefore the non-empty locally closed subset $f^{-1}(y)\cap U$ contains a point $x$ that is closed in $f^{-1}(y)$.

Because $y$ is closed in $Y$ and $f$ is locally of finite type, the set $f^{-1}(y)$ is closed in $X$; moreover, a point that is closed in the closed subset $f^{-1}(y)$ is also closed in $X$. Thus $x$ is closed in $X$, and $U$ contains the closed point $x$, proving that $X$ is Jacobson.
\end{proof}

\section{Outlook}

\subsection{Formal Completions}

\begin{proposition}\label{prop:formal-good}
Let $f:X\to Y$ be a (weak) good morphism of schemes, $X^{\prime}\subseteq X$ and $Y^{\prime}\subseteq Y$ closed subschemes with $f(X^{\prime})\subseteq Y^{\prime}$. Then the induced map on formal completions $\widehat{f}:\widehat{X}_{X^{\prime}}\rightarrow\widehat{Y}_{Y^{\prime}}$ is (weak) good.
\end{proposition}

\begin{proof}
Consider the restriction $f|_{X^{\prime}}:X^{\prime}\to Y^{\prime}$. Since $X^{\prime}$ is a closed (hence locally closed) subscheme of $X$, by Proposition \ref{prop:good-basic}(1) the map $f|_{X^{\prime}}$ is (weak) good. Moreover, because $f(X^{\prime})\subseteq Y^{\prime}$, we may view $f|_{X^{\prime}}$ as a map $X^{\prime}\to Y^{\prime}$ and apply Proposition \ref{prop:good-basic}(2) to see that it remains (weak) good when regarded as a map into the subspace $Y^{\prime}$.

The underlying topological space of the formal completion $\widehat{X}_{X^{\prime}}$ is homeomorphic to the underlying topological space of $X^{\prime}$; similarly $\widehat{Y}_{Y^{\prime}}$ is homeomorphic to $Y^{\prime}$. Under these homeomorphisms, the map $\widehat{f}$ corresponds precisely to $f|_{X^{\prime}}:X^{\prime}\to Y^{\prime}$. Because (weak) goodness is a purely topological condition, $\widehat{f}$ inherits the (weak) good property from $f|_{X^{\prime}}$.
\end{proof}

\subsection{A General Pattern}

The applications follow a pattern:
\begin{enumerate}
    \item Find a natural class of maps that are good (e.g., affine space projections).
    \item Use properties of good maps to show other maps are good or weak good.
    \item Derive geometric consequences.
\end{enumerate}
This suggests similar developments in other non-Hausdorff geometric categories, such as formal schemes or rigid analytic spaces, by finding appropriate analogues of affine space projections.

\subsection{Extension to Non-Archimedean Geometry}\label{sec:non-archimedean}

Because the definition of a good map is purely topological and the key local-on-source/target properties only require sobriety, the theory developed here can be transferred directly to \textbf{adic spaces} and \textbf{perfectoid spaces}, which are sober topological spaces. A natural next step is to investigate whether important classes of morphisms in rigid-analytic or perfectoid geometry satisfy the good property. For instance:
\begin{itemize}
    \item Are \textbf{finite-type morphisms} of adic spaces good?
    \item Are \textbf{pro-\'etale morphisms} or \textbf{diamonds} good?
    \item Does an analogue of Theorem \ref{thm:main-good} hold in the category of perfectoid spaces?
\end{itemize}
Positive answers would yield topological generalizations of Chevalley-type results in non-Archimedean geometry. Moreover, the stability of goodness under formal completion (Proposition \ref{prop:formal-good}) suggests a natural bridge between good maps and formal geometry in the sense of Scholze-Weinstein.

\end{document}